\documentclass[12pt]{amsart}
\usepackage{amsmath}
\usepackage{amssymb}
\usepackage{amsthm}
\usepackage{array}
\usepackage{xy}
\usepackage[pdftex]{graphicx}
\usepackage{hyperref}
\usepackage{color}
\usepackage{transparent}
\usepackage{latexsym}

\numberwithin{equation}{section}

\def\RR{\mathbb{R}}
\def\SS{\mathbb{S}}
\def\NN{\mathbb{N}}

\def\SS{\mathbb{S}}

\def\to{\rightarrow}

\theoremstyle{plain}

\theoremstyle{plain}
\newtheorem{theorem}{Theorem} [section]

\newtheorem{lemma}[theorem]{Lemma}

\theoremstyle{definition}

\theoremstyle{remark}
\newtheorem{remark}[theorem]{Remark}

\numberwithin{theorem}{section}
\numberwithin{equation}{section}
\numberwithin{figure}{section}

\begin{document}
\title[]{Regularity of shadows and the geometry of the singular set associated to a Monge-Amp\`{e}re equation}
\author[E. Indrei and L. Nurbekyan]{E. Indrei and L. Nurbekyan}

\def\signei{\bigskip\begin{center} {\sc Emanuel Indrei\par\vspace{3mm}
MSRI\\  
17 Gauss Way\\
Berkeley, CA 94720\\
email:} {\tt eindrei@msri.org }
\end{center}}

\def\signln{\bigskip\begin{center} {\sc Levon Nurbekyan \par\vspace{3mm}
Center for Mathematical Analysis,\\
Geometry, and Dynamical Systems\\
Departamento de Matem\'atica\\
Instituto Superior T\'ecnico\\
Lisboa 1049-001, Portugal\\
email:} {\tt lnurbek@math.ist.utl.pt}
\end{center}}


\makeatletter
\def\blfootnote{\xdef\@thefnmark{}\@footnotetext}
\makeatother

\blfootnote{The first author acknowledges support from the Australian Research Council and US NSF grant DMS-0932078, administered by the Mathematical Sciences Research Institute in Berkeley, California. The second author acknowledges support from the department of mathematics at the University of Texas at Austin and the Centre for Mathematical Analysis, Geometry, and Dynamical Systems at Instituto Superior T\'ecnico.}

\date{}

\maketitle

\begin{abstract}
Illuminating the surface of a convex body with parallel beams of light in a given direction generates a shadow region. We prove sharp regularity results for the boundary of this shadow in every direction of illumination. Moreover, techniques are developed for investigating the regularity of the region generated by orthogonally projecting a convex set onto another. As an application we study the geometry and Hausdorff dimension of the singular set corresponding to a Monge-Amp\`{e}re equation.
\end{abstract}

\section{Introduction}
Shadows play an important role in many different branches of mathematics such as differential geometry,  convex analysis, geometric combinatorics, and functional analysis \cite{Gh, Cho, Bl, GS, Mar, Be, Ki, Arn}. Our aim in this paper is to show that they also naturally appear in a free boundary problem associated to a Monge-Amp\`ere equation. Indeed, it turns out that the regularity of certain shadow regions yields information on the Hausdorff dimension of the singular set appearing in the optimal partial transport problem \cite{CM, AFi, AFi2, I, CI}.

\subsection{Illumination shadows}
Illumination shadows form powerful tools in the classification of surfaces. For instance, it is a well-known fact that if every shadow boundary generated by parallel illumination on a Blaschke surface embedded in $\RR^3$ is a plane curve, then the surface is quadric \cite{Nom}. Moreover, in $2001$ Ghomi \cite{Gh} solved the shadow problem formulated in $1978$ by Wente: if $M$ is a closed oriented $2$-dimensional manifold and $f: M \rightarrow \RR^3$ is a smooth immersion, then $f$ is a convex embedding if and only if the shadow region generated by parallel illumination is simply connected in every direction. 

Regularity properties of shadow boundaries have been investigated in \cite{Gh, St, GS, Z, Hern, How}. For example, given a smooth manifold it is well known that if the Gaussian curvature does not vanish at a given point, then the shadow boundary is locally smooth around that point (via the inverse function theorem). Moreover, using Sard's theorem, it is not difficult to prove that in almost every direction (in the sense of Lebesgue), the shadow boundary of a surface is continuous. Steenarts \cite{St} showed that for a smooth convex body, the shadow boundary has finite $(n-2)$-dimensional Hausdorff measure in almost all directions (in the sense of Lebesgue). On the other hand, Gruber and Sorger \cite{GS} showed that when one considers the product space of convex bodies and directions $\mathcal{G} \times \mathbb{S}^{n-1}$, most pairs $(\Omega, u) \in \mathcal{G} \times \mathbb{S}^{n-1}$ (i.e. up to a meagre set in the sense of Baire) generate boundaries with infinite $(n-2)$-dimensional Hausdorff measure while having Hausdorff dimension $(n-2)$, see \cite[Theorems 1 \& 2]{GS}. These results suggest that in the general class of convex bodies, most shadow boundaries are highly irregular.  

Moreover, in applications one may need regularity results for the shadow boundary of a generic convex body in every direction. In this weak setting, the tools of differential geometry do not apply due to lack of regularity and ``almost everywhere" results in the sense of Lebesgue or Baire do not suffice since they may exclude a dense set of directions. In \S \ref{PaI}, we address this problem with tools from convex analysis and obtain the following results.\\

\begin{itemize}

\item{[Theorem \ref{thm: gammacontinuous}] For a strictly convex domain in $\RR^n$, the boundary of the shadow generated by parallel illumination is locally a continuous graph in every direction.}  
 
\vskip .1in 
 
\item{[\S \ref{sh2}] There exists a convex set and a direction so that the shadow boundary generated by parallel illumination is not locally a graph. In particular, one may not remove the strict convexity assumption in Theorem \ref{thm: gammacontinuous}.}   

\vskip .1in

\item{[Theorem \ref{hold}] For a uniformly convex $C^{1,\alpha}$ domain in $\RR^n$, $\alpha \in (0,1]$, the boundary of the shadow generated by parallel illumination is locally a $C^{0,\alpha}$ graph in every direction.}

\vskip .1in
 
\item{[\S \ref{sh1}] For every $\alpha \in (0,1]$, there exists a $C^\infty$ smooth convex set and a direction, so that the shadow boundary generated by parallel illumination in that direction is $C^\beta$ for $\beta<\alpha$. In particular, one may not remove the uniform convexity assumption in Theorem \ref{hold}.}
 
\vskip .1in

\item{[Remark \ref{sh3}] For a uniformly convex $C^{k+1}$ domain in $\RR^n$, $k\ge 1$, the boundary of the shadow generated by parallel illumination is locally a $C^k$ graph.}\\           

\end{itemize}

We note that shadows generated by another type of illumination process also appear in a well-known covering problem of Levi \cite{Lev} and Hadwiger \cite{Ha}: let $\Omega \subset \RR^n$ be a convex body and $h(\Omega)$ the smallest number so that $\Omega$ can be covered by $h(\Omega)$ smaller homothetical copies of itself; the conjecture states that $h(\Omega) \le 2^n$, with equality if and only if $\Omega$ is an $n$-parallelotope. Indeed, Boltyanskii \cite{Bolt} connected this conjecture with an illumination problem by showing that $h(\Omega)=l(\Omega)$ where $l(\Omega)$ is the smallest number of light sources outside of $\Omega$ required to illuminate $\partial \Omega$; a boundary point $y$ of $\Omega$ is said to be illuminated from $x \notin \Omega$ if the line through $x$ and $y$ intersects the interior of $\Omega$. For further reading, we refer the reader to two survey articles \cite{Be, Mar} and the references therein.

\subsection{Projection shadows}
In $1986$ Kiselman \cite{Ki} addressed the following question: what degree of smoothness does a two-dimensional projection of a three-dimensional smooth convex set possess? He proved that if the convex set is $C^1$, then its projection is also $C^1$; if it is $C^{2,1}$, then the boundary of the projection is twice differentiable; and, if it is real-analytic, then the boundary of the projection is $C^{2,\alpha}$ for some $\alpha>0$. Moreover, he provided examples to show that these results are essentially sharp: in the real-analytic case, the boundary of the projection may be exactly $C^{2,\frac{2}{q}}$ for any odd integer $q \ge 3$ \cite[Example  3.2]{Ki}, and the boundary of the shadow of a $C^\infty$ set may not be $C^2$ \cite[Example 3.3]{Ki}. V. Sedykh \cite{Sed} studied this question in higher dimensions and proved that the projection of a smooth closed convex surface in $\RR^{n}$ onto a hyperplane is $C^{1,1}$ and showed that this result is sharp in the sense that there exists a hypersurface whose shadow is not twice differentiable; this contrasts with Kiselman's result in $\RR^3$. Moreover, the analytic case also displays a loss of regularity in higher dimensions: Bogaevsky \cite{Bog}  showed the existence of a real-analytic closed convex hypersurface, whose shadow does not belong to the class $C^2$. These results are all compiled and discussed in the book ``Arnold's problems" by V.I. Arnold \cite{Arn} (Arnold calls these types of shadows ``apparent contours"). 

In applications, however, one may require regularity results of this shadow when projecting onto a strictly convex domain (as opposed to a hyperplane as in the results above), see e.g. \S \ref{MAP}. Indeed, this situation is quite different in the sense that the projection no longer occurs in just one direction, but in many different directions determined by the normal of the set onto which the projection takes place. Therefore, the variation of this normal dictates the regularity and geometry of the boundary of the projection and this requires a new approach in contrast with the affine case. Here is the precise statement of the problem: given two convex domains $\Omega \subset \RR^n$, $\Lambda \subset \RR^n$, if $P_\Lambda(\Omega)$ denotes the orthogonal projection of $\Omega$ onto $\Lambda$, then how smooth is $\partial(P_\Lambda(\Omega)\cap \partial \Lambda)$? The following results are established in \S \ref{convex shadow}.\\

\begin{itemize}

\item{[Theorem \ref{wc}] Let $\Omega \subset \mathbb{R}^n$ be a bounded strictly convex domain and $\Lambda \subset \mathbb{R}^n$ a convex domain whose boundary is $C^{1,1}$. If $\overline{\Omega} \cap \overline{\Lambda} = \emptyset$, then $\partial P_\Lambda(\Omega)$ is finitely $(n-2)$-rectifiable.}

\vskip .1in

\item{[Remark \ref{cant}]} The disjointness assumption in Theorem \ref{wc} is necessary: there exist two bounded convex domains $\Omega$ and $\Lambda$ in $\RR^2$ for which $\mathcal{H}^{0}(\partial(P_\Lambda(\Omega)\cap \partial \Lambda))=\infty$. 

\vskip .1in

\item{[Theorem \ref{sc}] If $\Omega$ and $\Lambda$ are $C^{k+1}$ convex domains in $\mathbb{R}^n$ with disjoint closures, $k\ge1$, and $\Omega$ is bounded and uniformly convex, then $\partial P_\Lambda(\Omega)$ is an $(n-2)$-dimensional $C_{loc}^k$ hypersurface.}\\

\end{itemize}

We point out that when one takes $\Lambda$ to be a hyperplane, Theorem \ref{wc} is immediate: the projection of a convex set onto a hyperplane is convex, so $\partial P_\Lambda(\Omega)$ is locally Lipschitz. However, the situation is different if $\Lambda$ is curved. Here is the idea of our method: we take a point $y \in \partial P_\Lambda(\Omega)$ and represent $\Lambda$ locally by a bi-Lipschitz graph with respect to the tangent space at $y$, $\mathbb{T}_y \Lambda = \mathbb{R}^{n-1}$. Then we consider $P_{\mathbb{R}^{n-1}} \partial P_\Lambda(\Omega)$ and cook up an auxiliary uniformly convex $C^{1,1}$ function that touches this set at $P_{\mathbb{R}^{n-1}}\partial P_\Lambda(y)$. By applying our results from Theorem \ref{hold} (or rather, the idea in the proof), we show that there exists a Lipschitz function which touches $P_{\mathbb{R}^{n-1}}\partial P_\Lambda(\Omega)$ at $P_{\mathbb{R}^{n-1}}\partial P_\Lambda(y)$ and bounds $P_{\mathbb{R}^{n-1}}\partial P_\Lambda(\Omega)$ from one side (in a suitable coordinate system). This yields the existence of a cone whose opening can be shown to depend only on the initial data (i.e. $\Omega$ and $\Lambda$) and that touches $P_{\mathbb{R}^{n-1}}\partial P_\Lambda(\Omega)$ only at $P_{\mathbb{R}^{n-1}}\partial P_\Lambda(y)$; the rest follows by iterating the argument above and locally transporting cones at all the other points in $P_{\mathbb{R}^{n-1}} \partial P_\Lambda(\Omega)$ from the surrounding tangent spaces via the $C^{1,1}$ charts representing $\Lambda$ and applying a standard covering argument from geometric measure theory. 

The idea of this argument in terms of finding a cone was employed by Indrei \cite{I}, although he assumed $\Lambda$ to be uniformly convex. The novelty in this paper is that we construct our barrier-type function without requiring uniform convexity of $\Lambda$. Indeed, this support function is constructed by using the boundary of the shadow generated by illuminating $\Lambda$ in the direction of some normal of $\Omega$ at the point $y + t(y)N_\Lambda(y)$, where $t(y)$ is the first hitting time of $\Omega$. However, in contrast with \cite[Proposition 4.9]{I}, we require a strict convexity assumption on $\Omega$. Nevertheless, this tradeoff turns out to be more useful when applying our theory to a free boundary problem that has a strict convexity assumption on $\Omega$ naturally built into it, see \S \ref{MAP}.          

On the other hand, the method we employ to prove Theorem \ref{sc}  is completely different. The starting point is that we may represent the sets $\partial \Omega$ and $\partial \Lambda$ locally as level sets of two convex functions functions $G:\RR^n \rightarrow \RR$ and $F: \RR^n \rightarrow \RR$. By exploiting the geometry of the problem, we construct a function $\phi: \RR^{2n+1} \rightarrow \RR^{n+3}$ so that $\partial P_\Lambda(\Omega)$ is locally a level set of $\phi$ (herein lies the novelty of our approach since we are connecting the two sets and the unknown shadow boundary by a single function); next, we compute the differential of this map and show that it has full rank and conclude via the implicit function theorem. 

\subsection{Shadows and a Monge-Amp\`{e}re equation} \label{MAP}

The optimal partial transport problem is a generalization of the classical Monge-Kantorovich problem: given two non-negative functions $f=f\chi_\Omega, \hskip .1in g=g\chi_\Lambda \in L^1(\mathbb{R}^n)$ and a number $0<m \leq \min\{||f||_{L^1}, ||g||_{L^1}\},$ the objective is to find an optimal transference plan between $f$ and $g$ with mass $m$. A transference plan is a non-negative, finite Borel measure $\gamma$ on $\mathbb{R}^n \times \mathbb{R}^n$, whose first and second marginals are controlled by $f$ and $g$ respectively: for any Borel set $A \subset \mathbb{R}^n$,
$$\gamma(A\times \mathbb{R}^n) \leq \int_{A} f(x)dx, \hskip .2in \gamma(\mathbb{R}^n \times A) \leq \int_{A} g(x)dx.$$
An optimal transference plan is a minimizer of the functional  
\begin{equation} \label{mini}
\gamma \rightarrow \int_{\mathbb{R}^n \times \mathbb{R}^n} c(x,y) d\gamma(x,y),
\end{equation}       
where $c$ is a non-negative cost function. Issues of existence, uniqueness, and regularity of optimal transference plans have been addressed by Caffarelli \& McCann \cite{CM}, Figalli \cite{AFi, AFi2}, Indrei \cite{I}, and Chen \& Indrei \cite{CI}.  

If $$||f \wedge g||_{L^1(\mathbb{R}^n)}\leq m \leq \min\{||f||_{L^1(\mathbb{R}^n)}, ||g||_{L^1(\mathbb{R}^n)}\} \large,$$ then by the results in \cite[Section 2]{AFi}, there exists a convex function $\Psi_m$ and non-negative functions $f_m$,  $g_m$ for which $$\gamma_m:=(Id \times \nabla \Psi_m)_{\#}f_m=(\nabla \Psi_m^* \times Id)_{\#}g_m,$$ is the unique solution of (\ref{mini}) and $\nabla {\Psi_m}_{\#}f_m=g_m$ (see \cite[Theorem 2.3]{AFi}).

$\Psi_m$ is known as the \textit{Brenier solution} of the Monge-Amp\`{e}re equation $$\operatorname{det}(D^2\Psi_m)(x)=\frac{f_m(x)}{g_m(\nabla \Psi_m(x))},$$ with $x \in F_m:=$ set of density points of $\{f_m>0\},$ and $\nabla \Psi_m(F_m) \subset G_m$:= set of density points of $\{g_m>0\}.$ Moreover, as in \cite[Remark 3.2]{AFi}, we set $$U_m:=(\Omega \cap \Lambda) \cup \bigcup_{(\bar x, \bar y) \in \Gamma_m} B_{|\bar x - \bar y|}(\bar y),$$

$$V_m:=(\Omega \cap \Lambda) \cup \bigcup_{(\bar x, \bar y) \in \Gamma_m} B_{|\bar x - \bar y|}(\bar x),$$ where $\Gamma_m$ is the set $$(Id \times \nabla \Psi_m)(F_m \cap D_{\nabla \Psi_m}) \cap (\nabla \Psi_m^* \times Id)(G_m \cap D_{\nabla \Psi_m^*}),$$ with $D_{\nabla \Psi_m}$ and $D_{\nabla \Psi_m^*}$ denoting the set of continuity points for $\nabla \Psi_m$ and $\nabla \Psi_m^*$, respectively, where $\Psi_m^*$ is the Legendre transform of $\Psi_m$.    

The free boundary associated to $f_m$ is denoted by $\overline{\partial U_m \cap \Omega}$ and the free boundary associated to $g_m$ by $\overline{\partial V_m \cap \Lambda}$. They correspond to $\overline{\partial F_m \cap \Omega}$ and $\overline{\partial G_m \cap \Lambda}$, respectively \cite[Remark 3.3]{AFi}. One method of obtaining free boundary regularity is to first prove regularity results on $\Psi_m$ and then utilize that $\nabla \Psi_m$ gives the direction of the normal to the free boundary $\overline{\partial U_m \cap \Omega}$ \big(by symmetry and duality, this also implies a similar result for $\overline{\partial V_m \cap \Lambda}$\big). 

Indeed, this method was employed by Caffarelli \& McCann \cite{CM} to deduce $C_{loc}^{1,\alpha}$ free boundary regularity away from a singular set $\tilde S$ in the case when $\Omega$ and $\Lambda$ are strictly convex and separated by a hyperplane. Indrei \cite{I} generalized an improvement of this result in the overlapping case: he obtains $C_{loc}^{1,\alpha}$ free boundary regularity away from the common region $\Omega \cap \Lambda$ and a singular set $S$ which in the disjoint case is a subset of $\tilde S$. Moreover, he developed a method to study the Hausdorff dimension of $\tilde S$ and utilized it to prove that if the domains are $C^{1,1}$ and uniformly convex, then $S$ has Hausdorff dimension $(n-2)$.      

In \S \ref{APP}, we connect the shadow boundaries with this singular set and show that one may replace the uniform convexity assumption with a strict convexity assumption to obtain that the singular set has Hausdorff dimension $(n-2)$, see Theorem \ref{singset}. The precise connection is this: the singular set breaks up into two parts; one of these can be handled using notions from transport theory and non-smooth analysis; the other can be shown to be trapped on the boundary of $P_\Lambda(\Omega)$. Thus, understanding the Hausdorff dimension of the boundary of this shadow is a way to obtain bounds on the Hausdorff dimension of the singular set. This is where the rectifiability result of Theorem \ref{wc} comes into play. Since Theorem \ref{hold} was used in the proof of Theorem \ref{wc}, this highlights the interplay between the shadow generated by parallel illumination, the shadow generated by orthogonal projections, and the Monge-Amp\`ere free boundary problem arising in optimal transport theory.

\section{Regularity of shadows generated by parallel illumination} \label{PaI}
In this section we investigate the regularity of the shadow region of a convex domain $\Lambda \subset \mathbb{R}^n$ under parallel illumination. For $u \in \SS ^{n-1}$, we denote the shadow of $\Lambda$ by the set $S_u$ of points $x \in \partial \Lambda$ such that there exists a normal vector $\nu(x)$ (i.e. a vector in the normal cone of $\Lambda$ at $x$) for which $\langle \nu(x), u \rangle >0$. Our aim is to prove that for a strictly convex domain $\Lambda$, the boundary $\partial S_u$ (in the topology of $\partial \Lambda$) is locally a continuous graph and that this regularity is optimal in the sense that if $\Lambda$ is not strictly convex, then $\partial S_u$ might fail to locally be a graph.

Given $k \in \NN$ and $x=(x_1,x_2, \cdots, x_k) \in \RR ^k,$ we denote an arbitrary vector in $\RR^{k-1}$ by $x':=(x_1, x_2, \cdots, x_{k-1}).$ Furthermore, let $x'':=(x')'=(x_1, x_2, \cdots, x_{k-2}) \in \RR ^{k-2}$. For a set $A \subset \RR ^k$ define $A':=\{x' \ : \ x \in A\}$ and $A'':=(A')'$.

Let $x \in \partial \Lambda$ be a boundary point of the shadow $S_u$. Without loss of generality we may assume that in a neighborhood of $x$, say $U$, $\Lambda$ is parametrized as $x_n \leq \phi (x')$, for some strictly concave function $\phi$. Consequently, $\partial \Lambda$ is locally given by $x_n = \phi (x')$ where the domain of $\phi$ is $U'\subset \RR^{n-1}$. Note that $x' \mapsto (x',\phi(x'))$ is a homeomorphism between the spaces $U'$ and $\partial \Lambda \cap U$.

For every $y' \in U'$ there is a one-to-one correspondence between superdifferentials $w \in \partial^{+} \phi (y')$ and normals $\nu$ at $(y', \phi(y'))$ given by $\nu = \frac{(-w,1)}{(|w|^2+1)^{1/2}}$. Therefore $(y', \phi(y')) \in S_u$ if and only if $\langle w(y'), u' \rangle < u_n$, for some $w \in \partial^{+} \phi (y')$.

In this section we prove that $S'_u \cap U'$ (in the usual $\RR ^{n-1}$ topology) is locally a continuous graph. By rotating the coordinate system, if necessary, we may assume $x=0$ and $u'=(0,0, \cdots, 1) \in \RR ^{n-1}$; moreover, we identify $\RR^{n-2}$ with $(u')^\perp$. Under these assumptions, the condition $\langle w, u' \rangle <u _n$ takes the form $w_{n-1} < u_n$. We begin our analysis with the following lemma.
\begin{lemma}
Let $\Lambda \subset \mathbb{R}^n$ be a strictly convex domain and $y'\in S'_u \cap U'$. Then $(y'',\alpha) \in S'_u$, for every $\alpha > y_{n-1}$ such that $(y'',\alpha) \in U'$.
\end{lemma}
\begin{proof}
Since $y' \in S'_u$, there exists $w^1 \in \partial ^{+} \phi (y')$ such that $w^1_{n-1} < u_n$. Let $w_2 \in \partial ^{+} \phi (y'',\alpha)$ be any element in the superdifferential. By the monotonicity formula,  $$\langle w^2-w^1, (y'', \alpha)-y' \rangle <0,$$ or equivalently $(w^2_{n-1}-w^1_{n-1})(\alpha-y_{n-1})<0$. Therefore, $w^2_{n-1}<w^1_{n-1}$. Combining this with $w^1_{n-1}<u_n$ yields $w^2_{n-1}<u_n$, and this implies $(y'',\alpha) \in S'_u$.
\end{proof}

\begin{lemma}\label{lemma2}
Let $\Lambda \subset \mathbb{R}^n$ be a strictly convex domain. Then there exists a ball $V'' \subset \RR^{n-2}$ centered at $0''$ with the following properties: for every $y'' \in V''$, there exist $\alpha, \beta \in \RR$ such that $(y'',\alpha),(y'',\beta) \in U'$ and for every $\eta \in \partial ^{+} \phi (y'',\alpha)$ and $\zeta \in \partial ^{+} \phi (y'',\beta)$, one has $\eta _{n-1}<u_n<\zeta _{n-1}$.
\end{lemma}
\begin{proof}
Since the set $\partial ^{+} \phi(0')$ is convex, one of the following is true:
  \begin{itemize}
    \item[(i)] $w_{n-1} < u_n$ for every $w \in \partial^{+} \phi (0')$;
    \item[(ii)] $w_{n-1} > u_n$ for every $w \in \partial^{+} \phi (0')$;
    \item[(iii)] $w_{n-1} = u_n$ for some $w \in \partial^{+} \phi (0')$.
  \end{itemize}
However, by continuity properties of the superdifferential of a convex function (see e.g. \cite[Corollary 24.5.1]{Ro}), if $\text{(i)}$ or $\text{(ii)}$ holds, then the strict inequality will be satisfied in some neighborhood of $0'$. This contradicts $0' \in \partial S'_u$. Hence, $w_{n-1} = u_n$ for some $w \in \partial^{+} \phi (0')$. Pick $\beta <0 < \alpha$ such that $(0'',\alpha), (0'',\beta) \in U'$. The monotonicity formula implies that every $\eta \in \partial^{+} \phi(0'',\alpha)$ satisfies $\eta _{n-1} < w_{n-1}=u_n$, and similarly every $\zeta \in \partial ^{+} \phi(0'',\beta)$ satisfies $\zeta _{n-1} > w_{n-1}=u_n$. By utilizing the continuity of the superdifferential again, there exists a ball $V''$ centered at $0''$ such that for every $y'' \in V''$ with $(y'',\alpha),(y'',\beta) \in U'$, we have that for every $\eta \in \partial^{+} \phi(y'',\alpha)$ and $\zeta \in \partial ^{+} \phi(y'',\beta)$, $\eta _{n-1} < u_n<\zeta _{n-1}$.
\end{proof}

Let $V''$ be the ball from Lemma \ref{lemma2}. For every $y'' \in V''$, define
\begin{equation}\label{eq: gamma}
  \gamma(y''):= \inf \{ t \ : \ (y'',t) \in S'_u \cap U' \}.
\end{equation}
By Lemma \ref{lemma2}, $\gamma$ is well-defined with $(y'',\gamma(y'')) \in U'$. 

\begin{lemma}(Properties of $\gamma$) \label{lemma3}
Let $\Lambda \subset \mathbb{R}^n$ be a strictly convex domain and $V''$ the ball from Lemma \ref{lemma2}. For every $y'' \in V''$ there exists $w \in \partial ^{+}\phi(y'',\gamma(y''))$ such that $w_{n-1}=u_n$. Moreover, if $t > \gamma(y'')$ and $(y'',t) \in U'$, then $\zeta_{n-1} < u_n$ for every $\zeta \in \partial ^{+}\phi(y'',t)$. Similarly, if $t < \gamma(y'')$ and $(y'',t) \in U'$, then $\zeta_{n-1} > u_n$ for every $\zeta \in \partial ^{+}\phi(y'',t)$. Furthermore, $\gamma$ is continuous.

\end{lemma}
\begin{proof}
Suppose $w_{n-1}<u_n$ for all $w \in \partial ^{+}\phi(y'',\gamma(y''))$. By continuity of the superdifferential, $\zeta _{n-1}<u_n$ for all $\zeta \in \partial ^{+}\phi(y'',t)$ if $t$ is sufficiently close to $\gamma(y'')$; therefore, $(y'',t) \in S'_u \cap U'$ for some $t<\gamma(y'')$, and this contradicts the definition of $\gamma$. On the other hand, if $w_{n-1}>u_n$ for all $w \in \partial ^{+}\phi(y'',\gamma(y''))$, then again by continuity of the superdifferential, $\zeta _{n-1}>u_n$ for all $\zeta \in \partial ^{+}\phi(y'',t)$ if $t$ is sufficiently close to $\gamma(y'')$. This implies $(y'',t) \in (S'_u)^c \cap U'$ for $\gamma(y'') < t \leq t_0$, with $t_0$ sufficiently close to $\gamma(y'')$; again, this produces a contradiction. Therefore, there exist $w^0,w^1 \in \partial^{+} \phi(y'',\gamma(y''))$ such that $w^1_{n-1} \leq u_n$ and $w^0_{n-1} \geq u_n$. Hence, for some $s \in [0,1]$, $w_{n-1}=u_n$ where $w=(1-s)w^0+s w^1$. Since $\partial^{+} \phi(y'',\gamma(y''))$ is a convex set, it follows that $w \in \partial^{+} \phi(y'',\gamma(y''))$. Now pick any $t>\gamma(y'')$. By the monotonicity formula, for every $\zeta \in \partial^{+} \phi(y'',t)$, $\zeta _{n-1} < w_{n-1}=u_n$. Similarly, if $t<\gamma(y'')$, then $\eta _{n-1}>w_{n-1}=u_n$ for all $\eta \in \partial^{+} \phi(y'',t)$. Next let $y''_k, y'' \in V''$ and $y''_k \to y''$. Since $\{\gamma(y''_k)\}$ is a bounded sequence, every subsequence has a further subsequence that converges. Take such a subsequence and suppose it converges to, say, $t \in \mathbb{R}$. If $t>\gamma(y'')$, then by what has already been proved, we have that $\zeta _{n-1}<u_n$ for all $\zeta \in \partial^{+} \phi(y'',t)$. Therefore, by the continuity of the superdifferential, this condition is satisfied in some neighborhood of $(y'',t)$, but this contradicts the fact that $(y''_k,\gamma(y''_k)) \to (y'',t)$ (along this subsequence) and that there exists $w^k \in \partial^{+} \phi(y''_k,\gamma(y''_k))$ such that $w^k_{n-1}=u_n$. The case $t<\gamma(y'')$ may be excluded in the same manner. Hence, $t=\gamma(y'')$, and we proved that every subsequence of $\{\gamma(y''_k)\}$ admits a further subsequence converging to $\gamma(y'')$; this implies the continuity of $\gamma$.
\end{proof}

Now we have all the ingredients to prove the following theorem which may be seen as the first step towards investigating the regularity of the boundary of the shadow region.

\begin{theorem}\label{thm: gammacontinuous} Let $\Lambda \subset \mathbb{R}^n$ be a strictly convex domain and $u \in \mathbb{S}^{n-1}$. Then the boundary of the shadow region generated by parallel illumination in the direction $u$ is locally the graph of a continuous function. More precisely,
  \begin{equation}
    \partial S'_u \cap U' \cap \left(V'' \times \RR\right) = \{ (y'', \gamma(y'')) \ : \  y'' \in V''\}.
  \end{equation}
\end{theorem}
\begin{proof} Lemma \ref{lemma3} implies the continuity of $\gamma$ and that $(y'',\alpha) \in S_u'$ for every $\alpha>\gamma(y'')$ and $(y'',\alpha) \in (S_u')^c$ for every $\alpha<\gamma(y'')$, where $(y'',\alpha) \in U' \cap \left(V'' \times \RR\right)$.   
\end{proof}

If the convex domain to be illuminated is uniformly convex, then the shadow boundary is locally H\"older continuous under mild regularity assumptions. The next theorem quantifies this statement.      

\begin{theorem} \label{hold}
If $\Lambda \subset \RR^n$ is a uniformly convex $C^{1,\alpha}$ domain, $\alpha \in (0,1]$, then $\partial S'_u$ is locally a $C^{0,\alpha}$ graph.
\end{theorem}
\begin{proof}
From Theorem \ref{thm: gammacontinuous} it follows that $\partial S'_u$ is the graph of a continuous function $\gamma$ defined on the ball $V''$. Therefore, it suffices to show that $\gamma$ is H\"older continuous on $V''$. Lemma \ref{lemma3} implies that for every $y'' \in V'', \ \gamma(y'')$ is the only solution of the equation $\frac{\partial \phi}{\partial y_{n-1}}(y'',\alpha)=u_n$, hence
  \begin{equation} \label{smo}
    \frac{\partial \phi}{\partial y_{n-1}}(y'',\gamma (y''))=u_n,
  \end{equation}
(recall that $\phi$ is the local chart representing $\partial \Lambda$).  
Since $\Lambda$ is $C^{1,\alpha}$ and uniformly convex, $\phi$ is $C^{1,\alpha}$ and uniformly concave, i.e.
  \begin{equation}\label{eq: phiC1,1}
    |\nabla \phi(y')-\nabla \phi(z')| \leq L |y'-z'|^\alpha
  \end{equation}
  and
  \begin{equation}\label{eq: phiunifconcave}
    \langle \nabla \phi(y')-\nabla \phi(z'), y'-z' \rangle \leq -\theta |y'-z'|^2,
  \end{equation}
  for some $L,\theta >0$ and all $y',z' \in V'$. To prove that $\gamma$ is H\"older, it suffices to show that at every point on the graph of $\gamma$, we can place a cusp with uniform opening that stays above the graph. It suffices to prove it for one point since the proof is identical for any other point. Without loss of generality, we assume $0' \in \partial S'_u$ and show that a cusp can be placed at $0'$ that stays above the graph: fix a point $(y'',y_{n-1})$ such that $y_{n-1}>\frac{L}{\theta}|y''|^\alpha$. By \eqref{eq: phiC1,1} we have 
  \begin{equation*}
        \frac{\partial \phi}{\partial y_{n-1}}(y'',y_{n-1}) \leq \frac{\partial \phi}{\partial y_{n-1}}(0,y_{n-1})+L|y''|^\alpha.
  \end{equation*}
 On the other hand, the monotonicity formula \eqref{eq: phiunifconcave} and the assumption $0' \in \partial S'_u$ imply
  \begin{equation*}
        \frac{\partial \phi}{\partial y_{n-1}}(0,y_{n-1}) \leq \frac{\partial \phi}{\partial y_{n-1}}(0,0)-\theta y_{n-1}=u_n-\theta y_{n-1}.
  \end{equation*}
  By combining the previous two inequalities, it follows that
  \begin{equation*}
   \frac{\partial \phi}{\partial y_{n-1}}(y'',y_{n-1}) \leq u_n-\theta y_{n-1}+L|y''|^\alpha<u_n,
  \end{equation*}
which means $(y'',y_{n-1}) \in S'_u$ or equivalently $y_{n-1} > \gamma (y'')$. Thus, the epigraph of the cusp $y_{n-1} = \frac{L}{\theta} |y''|^\alpha$, i.e. $$\Big\{ (y'',y_{n-1}) \ : \ y_{n-1} > \frac{L}{\theta} |y''|^\alpha\Big\},$$ touches the graph of $\gamma$ from above.
\end{proof}
\begin{remark}
Note that the opening of the cusp in the proof of Theorem \ref{smo} is determined by $\frac{L}{\theta}$.
\end{remark}

\begin{remark} \label{sh3} If $\Lambda$ is a uniformly convex domain with a $C^{k+1},\ k \geq 1$, smooth boundary then it is not difficult to see that $\partial S'_u$ is a $C^k$ graph. Indeed, it suffices to show that $\gamma$ is a $C^k$ function. Since $\Lambda$ is uniformly convex, $\phi$ is uniformly concave or $D^2 \phi \leq -\theta$, for some $\theta>0$. In particular, $\frac{\partial ^2 \phi}{\partial y_{n-1}^2} \leq -\theta <0$. Since $\gamma (y'')$ is the only solution of the equation $\frac{\partial \phi}{\partial y_{n-1}}(y'',y_{n-1})=0$, by the implicit function theorem, $\gamma$ is as regular as $\frac{\partial \phi}{\partial y_{n-1}}$, i.e. $C^k$ (note that in the smooth case, we may assume without loss of generality that $u_n=0$ in (\ref{smo})).
\end{remark}

\subsection{Theorem \ref{thm: gammacontinuous} is sharp} \label{sh2} In Theorem \ref{thm: gammacontinuous}, it was shown that for a strictly convex set, the boundary of the shadow is locally a continuous graph in any given direction. It is natural to wonder if this result extends to merely convex sets. Indeed, the following counterexample shows that this is not so: in $\RR^3$, consider the circle $$\{(x,y,z): (x-1)^2+z^2 =1\}$$ and construct a cone-like set by connecting this circle to the point $(0,1,0)$ with line segments. It is not difficult to see that this process generates a convex body so that when it  is illuminated in the direction $(0,1,0)$, the boundary of the resulting shadow is $$\{(x,y,z): (x-1)^2+z^2 =1\} \cup \{(0,t,0): 0\le t \le 1\}.$$ In particular, the boundary of the shadow is not a graph near the origin in any coordinate system.

\subsection{Theorem \ref{hold} is sharp} \label{sh1}
Here, we show that Theorem \ref{hold} is optimal in the following sense: given a direction $u \in \mathbb{S}^{n-1}$, there exists a smooth (i.e. $C^\infty$) convex body $\Lambda \subset \mathbb{R}^n$ for which the boundary of the shadow is not locally H\"older continuous; therefore, the uniform convexity assumption is necessary for the conclusion of the theorem. Indeed, the key observation in the construction of the counterexample is that for a smooth, strictly convex set, $\gamma (y'')$ is the unique solution of the equation $\frac{\partial \phi}{\partial y_{n-1}}(y'',y_{n-1})=0$ (see (\ref{smo})), and since we are working locally, it suffices to find a smooth, strictly convex function $\phi : \RR^2 \rightarrow \RR$ whose level set $$\{(x,y): \partial_y \phi(x,y)=0\}$$ is far from smooth. In fact, an example like this already appeared in work of Kiselman \cite{Ki} in which the regularity of the projection of a three dimensional convex set onto a $2$-dimensional plane is analyzed (see \cite[Example 3.2]{Ki}): let $q$ be an odd natural 
number and set $$\phi(x,y)=x^2(4-y+\frac{1}{2}y^2)+\frac{1}{q+1}y^{q+1}-\frac{1}{q+2}y^{q+2};$$ note that $\phi$ is convex in the strip $|y|<\frac{1}{2}$ and $$\partial_y \phi(x,y)=(y^q-x^2)(1-y).$$ Thus, one may construct a smooth convex set in $\Lambda \subset \RR^3$ whose boundary is locally given by $\phi$ in a neighborhood of the origin. In particular, at the local level $$\{(x,y): \partial_y \phi(x,y)=0\}$$ is represented by $$\Big \{(x,y): y=|x|^{\frac{2}{q}}\Big \},$$  and by selecting $u=(0,1,0)$ it becomes evident that illuminating $\Lambda$ in the direction $u$ generates a shadow boundary which is of class $C^{\frac{2}{q}}$. Since $q$ can be taken arbitrarily large, this family of examples shows that for each $\alpha \in (0,1]$, there exists a smooth convex set $\Lambda_\alpha$ so that the boundary of the shadow is not $C^{0, \alpha}$. Note that this level set method also suggests a way of constructing shadows on the surface of convex bodies with a specified degree of regularity.

\section{Regularity of shadows generated by convex projections} \label{convex shadow}
Let $\Omega \subset \mathbb{R}^n$, $\Lambda \subset \mathbb{R}^n$ be two convex domains and suppose that we wish to orthogonally project  $\Omega$ onto $\Lambda$. This operation generates a shadow region $P_\Lambda(\Omega) \cap \partial \Lambda$ on the boundary of $\Lambda$. The purpose of this section is to study the regularity of this shadow. In other words, given $z_0 \in \partial(P_\Lambda(\Omega)\cap \partial \Lambda)$, we wish to understand how smooth $z_0 \in \partial(P_\Lambda(\Omega)\cap \partial \Lambda)$ is in a neighborhood of $z_0$. 

\subsection{Weak case}
\begin{theorem} \label{wc}
Let $\Omega \subset \RR^n$ be a bounded strictly convex domain and $\Lambda \subset \RR^n$ a convex domain whose boundary is $C^{1,1}$. If $\overline{\Omega}\cap \overline{\Lambda}=\emptyset$,  then $\partial P_\Lambda(\Omega)$ is finitely $(n-2)$-rectifiable.  
\end{theorem}

\begin{proof}
Consider an arbitrary $y \in \partial P_\Lambda(\Omega)$, and let $\phi: \mathbb{T}_y \Lambda \rightarrow \mathbb{R}$ be a local $C^{1,1}$ concave chart representing $\partial \Lambda$ in a neighborhood $B_{r_y}$ around $y$ so that $\nabla \phi(y)=0$; by translating the coordinate system, if necessary, we may also assume $y=0$. Note that the half-line $L$ at the origin in the direction of the normal of $\Lambda$ at $0$ touches $\Omega$ tangentially at some point, say, $x$ (since the projection occurs along the normal to $\Lambda$ and $y \in \partial P_\Lambda(\Omega)$). By convexity of $\Omega$, $L$ lives on a tangent space of $\Omega$ at $x$ with normal, say $\nu$. Since $\langle \nu, N_\Lambda(0)\rangle =0$, it follows that $\nu$ lives on the tangent space of $\Lambda$ at $0$. Let $e_{n-1}:= \nu$ and $\{e_1,\ldots,e_{n-1}\}$ be a basis for $\mathbb{R}^{n-1}$; set $$t_y^*:= dist(\Omega, \Lambda)+2 \max \{diam(\Omega),diam(B_{r_y})\},$$ and $$\Psi(z'):= \Psi(z'',z_{n-1}) = \phi(z')-\frac{1}{2t_y^*}|z|^2.$$ Note that $\Psi$ is $C^{1,1}$ and uniformly concave, so by Theorem \ref{hold}, it follows that locally around the origin, the level set
\begin{equation*} 
\{(z'',z_{n-1}): 0 = \partial_{z_{n-1}} \Psi(z'',z_{n-1})\},
\end{equation*}
is a Lipschitz graph which will be denoted by $\tilde \gamma(z'')=z_{n-1}(z'')$ (see (\ref{smo})). Now let $\gamma(z''):=\max\{\tilde \gamma(z''), 0\}$, and note that $\gamma$ is Lipschitz. We claim that locally around the orgin, 
\begin{equation} \label{pro} 
\Phi^{-1}(P_\Lambda(\Omega)) \subset \{(z'',z_{n-1}):  z_{n-1} \leq \gamma(z'')\},
\end{equation}    
where $\Phi(w):=(w',\phi(w'))$. Indeed, let $$z:=(z'',z_{n-1}) \in \Phi^{-1}(P_\Lambda(\Omega))\setminus \{0\};$$ if $z_{n-1}\le 0$, then since $\gamma \ge 0$, the result follows. So without loss of generality assume $z_{n-1}>0$. Since $z \in \Phi^{-1}(P_\Lambda(\Omega))$, it follows that $\Phi(z)+t(z)N_\Lambda(z) \in \partial \Omega$ where $t(z)>0$ is the first hitting time (the positivity follows from the disjointness assumption). Next, consider $$P_{\mathbb{R}^{n-1}}(\Phi(z)+t(z)N_\Lambda(z)) \subset \mathbb{R}^{n-1}$$ and note that the $e_{n-1}$ component of this point is negative (since $\Omega$ is strictly convex and $e_{n-1}$ is one of its outer normal vectors). In other words, $z_{n-1} - t(z) \partial_{z_{n-1}} \phi(z') < 0$. Thus, $\partial_{z_{n-1}} \phi(z') > 0$ (recall $z_{n-1}>0$) and since $t(z) \leq t_y^*$, it follows that 
\begin{equation} \label{rost}
\partial_{z_{n-1}} \Psi(z'',\tilde \gamma(z''))=0< \partial_{z_{n-1}} \phi(z')-\frac{1}{t_y^*}z_{n-1}=\partial_{z_{n-1}} \Psi(z'',z_{n-1});
\end{equation}
now assume by contradiction that $z_{n-1} > \gamma(z'')$. In particular, $z_{n-1} > \tilde \gamma(z'')$ so by monotonicity, 
$$\langle \nabla \Psi(z'',z_{n-1})- \nabla \Psi(z'',\tilde \gamma(z'')), (0,z_{n-1} - \tilde \gamma(z''))\rangle \leq 0.$$ Thus, $$\partial_{z_{n-1}} \Psi(z'',z_{n-1}) \le \partial_{z_{n-1}} \Psi(z'',\tilde \gamma(z''))=0,$$ and this contradicts (\ref{rost}) and proves the claim (i.e. (\ref{pro})). Next, note that $$0 \in \Phi^{-1}(P_\Lambda(\Omega)) \cap \{(z'',z_{n-1}):  z_{n-1} \leq \gamma(z'')\},$$ and since $\gamma$ is Lipschitz, (\ref{pro}) implies that we can place a cone oriented in the direction $e_{n-1}$ so that it lies in $$\RR^{n-1} \setminus \Phi^{-1}(P_\Lambda(\Omega)).$$ The opening of the cone depends on the Lipschitz constant of $\nabla \Psi$ and the uniform convexity constant of $-\Psi$; in particular, it depends on $t_y^*$. However, since $P_\Lambda(\Omega)$ is bounded (recall that $\Omega$ is bounded) and the domains have disjoint closures, it follows that $t_y^*$ has a uniform positive lower bound. The existence of this cone implies the claim within the proof of \cite[Proposition 4.1]{I}. Indeed, this is the only part where Indrei used the uniform convexity of $\Lambda$, which we were able to replace with strict convexity of $\Omega$ in our proof above; thus, the rest of the proof follows exactly as \cite[Proposition 4.1]{I} (the idea is that once we have a cone at a point, we can use the $C^{1,1}$ regularity to transition between charts to get a cone at every point of $\partial P_\Lambda(\Omega)$; nevertheless, the cones may be oriented in different directions, but this readily implies rectifiability via a covering argument). 
\end{proof}

\begin{remark}
If in Theorem \ref{wc} $\Lambda$ is bounded, then one may replace $t_y^*$ with $$t^*:= dist(\Omega, \Lambda)+2 \max \{diam(\Omega),diam(\Lambda)\}.$$ 
\end{remark}

\begin{remark} \label{cant}
The disjointness assumption in Theorem \ref{wc} is necessary: indeed, consider a Cantor set $\mathcal{C}$ on $[1,2]$ and let $g$ be a smooth function whose zero level set is $\mathcal{C}$. For $\epsilon>0$ small, it follows that $f(x):=x^2+ \epsilon g(x)$ is convex, so its epigraph is a convex set in $\mathbb{R}^2$. Moreover, consider the epigraph of the function $h(x):=x^2$; of course, it is likewise convex. Now it is not difficult to see that using these epigraphs, one may obtain two bounded convex sets, say $\Omega$ and $\Lambda$, with the property that their boundaries intersect on the image of $\mathcal{C}$ under $h$. In this case, $\partial(P_\Lambda(\Omega)\cap \partial \Lambda)$ does not have finite $\mathcal{H}^{0}$ measure. Nevertheless, in the general case one may still prove a local version of Theorem \ref{wc} away from $\partial(\partial(\Omega \cap \Lambda) \cap \partial \Lambda)$.    
\end{remark}

\subsection{Smooth case}
In Theorem \ref{wc}, we utilized a geometric method of investigating the regularity of shadow boundaries generated by orthogonal projections. In what follows, we develop a more functional approach to attack this problem. The idea is to represent the unknown boundary as the level set of a function defined in terms of local charts. However, since the differential of this function contains the information regarding the regularity of the level set, we need to ensure that this function is smooth enough; this leads us to impose higher regularity on the domains.

\begin{theorem} \label{sc} Suppose $\Omega \subset \RR^n$ and $\Lambda \subset \RR^n$ are $C^{k+1}$, $k\ge1$,  convex domains separated by a hyperplane with $\Omega$ bounded and uniformly convex. If $\overline{\Omega}\cap \overline{\Lambda}=\emptyset$, then $\partial P_\Lambda(\Omega)$ is locally a $C^k$ smooth $(n-2)$-hypersurface. 
\end{theorem}

\begin{proof}
Given a point $y \in \partial \Lambda$ let $f:\mathbb{R}^{n-1} \rightarrow \mathbb{R}$ be the $C^{k+1}$ concave function which represents $\partial \Lambda$ locally around $y$. Likewise, for $x \in \partial \Omega$ let $g$ denote the $C^{k+1}$ uniformly concave function locally representing $\Omega$ around $x$. Set $$F(y_1,\ldots,y_n)=y_n-f(y_1,\ldots,y_{n-1}),$$ $$G(x_1,\ldots,x_n)=x_n-g(x_1,\ldots,x_{n-1})$$ and consider the function $$\phi: \mathbb{R}^n \times \mathbb{R}^n \times \mathbb{R} \rightarrow \mathbb{R}^{n+3}$$ given by $$\phi(x,y,t):=(G(x), F(y), \nabla G(x) \cdot \nabla F(y), y+t\nabla F(y) - x).$$ Geometric considerations imply that locally $\partial P_\Lambda(\Omega) = \phi^{-1}(0,0,0,0)$: indeed, $\{F(y)=0\}$ locally describes the boundary of $\Lambda$ and $\{G(x)=0\}$ that of $\Omega$; if $\nabla G(x) \cdot \nabla F(y)=0$, then the normal of $\Lambda$ at $y$ is orthogonal to the normal of $\Omega$ at $x$, and this implies that $y=P_\Lambda(x)$ is a boundary point of $P_\Lambda(\Omega)$; note that in this case, $t=t(x,y)=|x-y|/|\nabla F(y)|$ and the positive separation implies $t>0$. Our goal is to investigate the differential of this map in order to apply the implicit function theorem. With this in mind, let 
\begin{align*}
\phi_1(x,y,t)&:=G(x)\\
\phi_2(x,y,t)&:=F(y)\\
\phi_3(x,y,t)&:=\nabla G(x)\cdot \nabla F(y)\\
\Phi(x,y,t)&:=[\phi_4(x,y,t),\ldots,\phi_{n+3}(x,y,t)]^T:=y + t\nabla F(y) - x. 
\end{align*}
Thus, 

\begin{align*}
\nabla_x \phi_1&=\nabla G(x)\\
\nabla_y \phi_1&=0\\
\partial_t \phi_1&=0\\
\nabla_x \phi_2&=0\\
\nabla_y \phi_2&=\nabla F(y)\\
\partial_t \phi_2&=0\\
\nabla_x \phi_3&=D^2G(x) \nabla F(y)\\
\nabla_y \phi_3&=D^2F(y) \nabla G(x)\\
\partial_t \phi_3&=0\\
D_x \Phi&=-Id \in \mathbb{R}^{n\times n}\\
D_y \Phi&=Id +tD^2F(y) \in \mathbb{R}^{n\times n}\\
D_t \Phi&= \nabla F(y). 
\end{align*}
Therefore, 
$$D\Phi(x,y,t) = \begin{bmatrix} \nabla G(x)^T & 0 & 0 \\ 0 & \nabla F(y)^T & 0 \\ (D^2 G(x) \nabla F(y))^T & (D^2 F(y)\nabla G(x))^T & 0 \\-Id & Id+tD^2 F(y) & \nabla F(y) \end{bmatrix}$$ (note that this is an $(n+3) \times (2n+1)$ matrix). The strategy now is to prove $\ker(D \Phi)^T = \{0\}$ at points $(x,y,t) \in \phi^{-1}(0,0,0,0)$. Indeed, let $$(\alpha_1,\alpha_2,\alpha_3,v) \in \ker(D\Phi)^T,$$ and note that since 
$$D\Phi(x,y,t)^T = \begin{bmatrix} \nabla G(x) & 0 & D^2G(x) \nabla F(y) & -Id\\ 0 & \nabla F(y) & D^2 F(y) \nabla G(x) & Id+tD^2F(y)\\ 0 & 0 & 0 & \nabla F^T(y) \end{bmatrix},$$ we have 

\begin{align}
0&=\alpha_1 \nabla G(x) + \alpha_3 D^2 G(x) \nabla F(y) - v;\label{one}\\
0&=\alpha_2 \nabla F(y) + \alpha_3 D^2F(y)\nabla G(x) + v + t D^2F(y) v  \label{two};\\
 0&=\nabla F(y) \cdot v.
\end{align}
In particular, 

\begin{align*}
0&=\nabla F(y) \cdot v\\
&=\nabla F(y) \cdot (\alpha_1 \nabla G(x) + \alpha_3 D^2 G(x) \nabla F(y))\\
&=\alpha_2 \nabla F(y) \cdot \nabla G(x)+\alpha_3 \nabla F^T(y) D^2 G(x) \nabla F(y)\\
&=\alpha_3 \nabla F^T(y) D^2 G(x) \nabla F(y),
\end{align*}
(note $\nabla F(y) \cdot \nabla G(x) = 0$ since $(x,y,t) \in \phi^{-1}(0,0,0,0)$). 
Since $G$ is uniformly convex, it follows that $\alpha_3=0$ and so (\ref{one}) implies $$v=\alpha_1 \nabla G(x);$$ plugging this information into (\ref{two}) and taking a dot product with $\nabla G(x)$ yields $$0=\alpha_1\big(t \nabla G(x)^T D^2 F(y) \nabla G(x) + |\nabla G(x)|^2\big).$$ Since $|\nabla G(x)|>0$, and $F$ is convex, it follows that $\alpha_1=0$ which readily implies $v=0$ and so $\alpha_2=0$. Thus, we proved $\ker(D \Phi^T) = \{0\}$; in particular, $rank(D \Phi) = n+3$ for each point of interest $(x,y,t)$. We may now use the implicit function theorem to conclude.                 
\end{proof}            

\begin{remark}
The disjointness assumption in Theorem \ref{sc} is necessary, cf. Remark \ref{cant}.  
\end{remark}

\section{The singular set associated to a Monge-Amp\`{e}re equation}   \label{APP}

In this section, a connection is established between the illumination shadow, the projection shadow, and the singular set associated to a Monge-Amp\`ere equation arising in mass transfer theory. More precisely, we apply the results of the previous sections to improve a result of Indrei \cite{I} (see \S \ref{MAP} for a description of the optimal partial transport problem and relevant notation). 

\subsection{The structure of the singular set} 
In order to analyze the singular set for the free boundaries, we recall two sets which play a crucial role in the subsequent analysis; cf. \cite[Equations (2.2) and (2.3)]{I}. The nonconvex part of the free boundary $\overline{\partial U_m \cap \Omega}$ is the closed set 
\begin{equation} \label{nc}
\partial_{nc} U_m:=\{x \in \overline{\Omega \cap U_m}: \Omega \cap U_m \hskip.1in \text{fails to be locally convex at $x$}\}.
\end{equation}  
Moreover, the nontransverse intersection points are defined by 
\begin{equation} \label{nt}
\partial_{nt} \Omega:= \{x \in \partial \Omega \cap \overline{\Omega \cap \partial U_m}: \langle \nabla \Psi_m(x)-x, z-x\rangle \leq 0 \hskip .1in \forall z \in \Omega \}, 
\end{equation}
where $\tilde \Psi_m$ is the extension of $\Psi_m$ given by \cite[Theorem 4.10]{AFi}. By duality, $\partial_{nc} V_m$ and $\partial_{nt} \Lambda$ are similarly defined. Now, for $x\in \partial (\Omega \cap U_m)$ let\\
\vskip .1in 
\noindent $L(x):=\Bigl\{\nabla \tilde \Psi_m(x)+ \frac{x-\nabla \tilde \Psi_m(x)}{|x-\nabla \tilde \Psi_m(x)|}t: t\geq0 \Bigr\};$\\
\noindent
\vskip .1in
\noindent $K:=\Bigl\{x \in \partial (\Omega \cap U_m): L(x) \cap \overline{\Omega \cap U_m} \subset \partial (\overline{\Omega \cap U_m}) \Bigr\};$\\
\vskip .1in
\noindent $S_1:=\nabla \tilde \Psi_m^{-1}(\partial_{nt} \Lambda) \cap K;$   \\
\vskip .1in  
\noindent $A_1:=S_1 \cap \partial U_m$; \\

\vskip .1in 
\noindent $A_2:=S_1\setminus \partial U_m.$\\  

\noindent The singular set of the free boundary $\overline{\partial V_m \cap \Lambda}$ is 
\begin{align*}
S &= (\nabla \tilde \Psi_m(\partial_{nc}U_m) \cup \nabla \tilde \Psi_m(S_1)) \cap \partial V_m \cap \partial \Lambda \large\\
&= (\nabla \tilde \Psi_m(\partial_{nc}U_m) \cap \partial V_m \cap \partial \Lambda \large) \cup (\nabla \tilde \Psi_m(S_1) \cap \partial V_m \cap \partial \Lambda) \\
&=(\nabla \tilde \Psi_m(\partial_{nc}U_m) \cap \partial V_m \cap \partial \Lambda \large) \cup \nabla \tilde \Psi_m(A_1) \cup \nabla \tilde \Psi_m(A_2), \\
\end{align*} 
see \cite[Theorem 4.9]{I}. 
The next lemma describes the first two sets appearing in $S$. 
\begin{lemma} \label{lehk}
Assume $\Omega \subset \mathbb{R}^n$ and $\Lambda \subset \mathbb{R}^n$ are strictly convex bounded domains with disjoint closures. Then $$(\nabla \tilde \Psi_m(\partial_{nc}U_m) \cap \partial V_m \cap \partial \Lambda \large) \cup \nabla \tilde \Psi_m(A_1)$$ is $\mathcal{H}^{n-2}$ $\sigma$-finite. Moreover, if $\Omega$ is $C^1$, then $$\mathcal{H}^{n-2}((\nabla \tilde \Psi_m(\partial_{nc}U_m) \cap \partial V_m \cap \partial \Lambda \large) \cup \nabla \tilde \Psi_m(A_1)) < \infty.$$   
\end{lemma}

\begin{proof}
For $y \in (\nabla \tilde \Psi_m(\partial_{nc}U_m) \cap \partial V_m \cap \partial \Lambda \large)$ set $x:=\nabla \tilde \Psi_m^*(y)$; since $\Omega$ is convex and $x\in \partial_{nc}U_m$, it follows that $x \notin \partial \Omega \setminus \partial U_m$. Moreover, since free boundary never maps to free boundary (see e.g. \cite[Proposition 2.15]{I}), we also have $x \notin \partial U_m \cap \Omega$, which implies $x \in \partial U_m \cap \partial \Omega$. Therefore, $$(\nabla \tilde \Psi_m(\partial_{nc}U_m) \cap \partial V_m \cap \partial \Lambda \large) \subset \nabla \tilde \Psi_m(\partial U_m \cap \partial \Omega)\cap \partial V_m \cap \partial \Lambda.$$ An application of \cite[Proposition 4.8]{I} yields that $$\nabla \tilde \Psi_m(\partial_{nc}U_m) \cap \partial V_m \cap \partial \Lambda \large$$ is $\mathcal{H}^{n-2}$ - finite; the fact that $\nabla \tilde \Psi_m(A_1)$ is $\mathcal{H}^{n-2}$ $\sigma$-finite ($\mathcal{H}^{n-2}$ finite if $\Omega$ is $C^1$) follows from \cite[Corollary 4.6]{I}.  
\end{proof}

In the following lemma, we establish a connection between the singular set $S$ and the boundary of the projection of $\Omega$ onto $\Lambda$ studied in  \S \ref{convex shadow}. 

\begin{lemma} \label{yay} Assume $\Omega \subset \mathbb{R}^n$ and $\Lambda \subset \mathbb{R}^n$ are strictly convex bounded domains with disjoint closures. Then 

\begin{equation} \label{dfg}
\nabla \tilde \Psi_m(A_2)\subset \partial P_\Lambda(\Omega).
\end{equation}
 
\end{lemma}

\begin{proof}
Let $y:=\nabla \tilde \Psi_m(x) \in \nabla \tilde \Psi_m(A_2)$, $L_t:=\nabla \tilde \Psi_m(x)+ \frac{x-\nabla \tilde \Psi_m(x)}{|x-\nabla \tilde \Psi_m(x)|}t$ and note that the half-line $\{L_t\}_{t\geq0}$ is tangent to the active region. Since $x \in \partial \Omega \setminus \partial U_m$, it follows that $L_t$ is tangent to $\Omega$ at $x$; hence, it is on a tangent space to $\Omega$ at $x$. Let $z=P_\Lambda(x) \in \partial \Lambda$ (recall that $P_\Lambda$ is the orthogonal projection operator). Then by the properties of the projection (and the convexity of $\Lambda$), $x-z$ is parallel to some normal $N_\Lambda(z)$ of $\Lambda$ at $z$. Since $x \in S_1$, it follows that $\nabla \tilde \Psi_m(x) \in \partial_{nt} \Lambda$; in particular, $x-\nabla \tilde \Psi_m(x)$ is parallel to $N_\Lambda(\nabla \tilde \Psi_m(x))$. Thus, by uniqueness of the projection, it readily follows that $z=\nabla \tilde \Psi_m(x)=y$. Combining $\{L_t\}_{t\geq0} \subset \mathbb{T}_x \Omega$ and $y=P_\Lambda(x)$ yields $y \in \partial P_\Lambda(\Omega)$.              
\end{proof}

Lemmas \ref{lehk} \& \ref{dfg} imply that the singular set $S$ is contained in the union of an $\mathcal{H}^{n-2}$ $\sigma$-finite set and $\partial P_\Lambda(\Omega)$ under a strict convexity and disjointness assumption on the domains. Thus, a way to obtain bounds on the Hausdorff dimension of the singular set is by studying the Hausdorff dimension of $\partial P_\Lambda(\Omega)$. In \cite[Proposition 4.1]{I}, Indrei shows that if $\Omega$ is a bounded convex domain and $\Lambda$ is uniformly convex, bounded, and $C^{1,1}$ smooth, then $P_\Lambda(\Omega) \cap \partial \Lambda$ is $(n-2)$-rectifiable away from $\partial(\partial(\Omega \cap \Lambda) \cap \partial \Lambda)$; in particular, if the domains have disjoint closures, then $$\mathcal{H}^{n-2}(\partial P_\Lambda(\Omega))<\infty.$$ The proof of \cite[Proposition 4.1]{I} is technical but relies on a simple idea which we describe in the language developed in this paper in order to further highlight the connection with shadows: let $y \in \partial P_\Lambda(\Omega)$ and $x \in \partial \Omega$ be such that $y=P_\Lambda(x)$. Then $\partial P_\Lambda(\Omega) \subset \partial \Lambda \setminus S_{N_\Omega(x)}$, where $N_\Omega(x)$ is any normal of $\Lambda$ at the point $x$ and $S_{N_\Omega(x)}$ is the shadow from \S \ref{PaI}. In other words, $P_\Lambda(\Omega)$ is trapped in the illuminated portion of $\partial \Lambda$ under parallel illumination in the direction $N_\Omega(x)$. Since $y \in \partial P_\Lambda(\Omega)\cap \partial S_{N_\Omega(x)}$, it follows that $\partial S_{N_\Omega(x)}$ acts as a one-sided support for $\partial P_\Lambda(\Omega)$ locally around $y$. Therefore, if one can place a cone in the shadow portion $S_{N_\Omega(x)}$, a compactness argument would yield the desired rectifiability result. Indeed, this is where the uniform convexity and $C^{1,1}$ assumptions come into play in the proof of \cite[Proposition 4.1]{I}. However, the results of \S \ref{PaI} shed new light on the regularity and uniform convexity assumptions; more specifically, they imply that by this method, the uniform convexity assumption is necessary to obtain the desired cone and the $C^{1,1}$ regularity assumption is irrelevant: indeed, \S \ref{sh1} shows that there exists a $C^\infty$ strictly convex set whose shadow is H\"older, but not Lipschitz (with arbitrarily small H\"older exponent). In particular, this shows that one may not hope to remove the uniform convexity assumption by the same method (i.e. by using $\partial S_{N_\Omega(x)}$ as a support function). However, Theorem \ref{wc} implies that one one may obtain the cone without a uniform convexity assumption; this is achieved by cooking up a new type of support function related to the distance between the two sets. Moreover, Theorem \ref{sc} yields a higher regularity result. With this discussion in mind, we obtain the following theorem which improves \cite[Theorem 4.9]{I}:

\begin{theorem} \label{singset} Assume $\Omega \subset \mathbb{R}^n$, $\Lambda \subset \mathbb{R}^n$ are bounded strictly convex domains and that $\Lambda$ has a $C^{1,1}$ boundary. If $\overline \Omega \cap \overline \Lambda = \emptyset$, then the free boundary $\overline{\partial V_m \cap \Lambda}$ is a $C_{loc}^{1,\alpha}$ hypersurface away from the compact, $\mathcal{H}^{n-2}$ $\sigma$-finite set: $$S:= (\nabla \tilde \Psi_m(\partial_{nc}U_m) \cap \partial V_m \cap \partial \Lambda \large) \cup \nabla \tilde \Psi_m(A_1) \cup \nabla \tilde \Psi_m(A_2).$$ If $\Omega$ has a $C^1$ boundary, then $S$ is $\mathcal{H}^{n-2}$ finite. Moreover, if $\Omega$ and $\Lambda$ are $C^{k+1}$, $k\geq1$, and $\Omega$ is uniformly convex, then $\nabla \tilde \Psi_m(A_2)$ is contained on an $(n-2)$-dimensional $C_{loc}^{k}$ hypersurface. 
\end{theorem}
 
\begin{remark}
By duality and symmetry, an analogous statement holds for $\overline{\partial U_m \cap \Omega}$.     
\end{remark}

\begin{remark}
One may remove the disjointness assumption and obtain corresponding results by utilizing the method in \cite[\S 4]{I}. 
\end{remark}

\emph{Acknowledgments.}
This work was completed while the first author was a Huneke Postdoctoral Scholar at the Mathematical Sciences Research Institute in Berkeley, California during the 2013 program ``Optimal Transport: Geometry and Dynamics," and while the second author was a Postdoctoral Fellow at the Instituto Superior T\'ecnico. The excellent research environment provided by the University of Texas at Austin, Australian National University, MSRI, and Instituto Superior T\'ecnico is kindly acknowledged.

\pagebreak

\signei

\signln

\end{document}